\documentclass[usletter, 10pt]{amsart}

\usepackage{amsthm,amsmath,amsfonts,amssymb}
\usepackage{amsrefs}

\usepackage[colorlinks,citecolor=blue,urlcolor=blue]{hyperref}
\usepackage{graphicx}

\usepackage{xspace}
\usepackage[capitalize,noabbrev]{cleveref}
\usepackage{nicefrac}
\usepackage{xcolor}
\usepackage{dsfont}
\usepackage{tikz}

\theoremstyle{plain}

\newtheorem{theorem}{Theorem}%
\newtheorem{lemma}{Lemma}

\newtheorem{proposition}[theorem]{Proposition}

\theoremstyle{definition}
\newtheorem{definition}{Definition}
\newtheorem{assumption}{Assumption}

\newtheorem*{remark}{Remark}

\crefname{axiom}{Axiom}{Axioms}
\crefname{assumption}{Assumption}{Assumptions}

\graphicspath{{./art/}}

\usepackage[plain,noend]{algorithm2e}

\makeatletter
\renewcommand{\algocf@captiontext}[2]{#1\algocf@typo. \AlCapFnt{}#2} %
\def\@algocf@capt@plain{top}
\renewcommand{\algocf@makecaption}[2]{%
  \addtolength{\hsize}{\algomargin}%
  \sbox\@tempboxa{\algocf@captiontext{#1}{#2}}%
  \ifdim\wd\@tempboxa >\hsize%
    \hskip .5\algomargin%
    \parbox[t]{\hsize}{\algocf@captiontext{#1}{#2}}%
  \else%
    \global\@minipagefalse%
    \hbox to\hsize{\box\@tempboxa}%
  \fi%
  \addtolength{\hsize}{-\algomargin}%
}
\makeatother

\def\v{{\varepsilon}}
\def\N{{\mathbb{N}}}
\def\Q{{\mathcal{Q}}}
\def\E{{\mathbb{E}}}

\newcommand{\GG}[1]{} %

\newcommand{\PTheta}{\mathcal{P}(\Theta)}

\DeclareMathOperator*{\argmin}{arg\,min}

\DeclareMathOperator*{\essinf}{ess\,inf}
\DeclareMathOperator*{\supp}{supp}

\newcommand{\acro}[1]{\textsc{#1}\xspace}
\newcommand{\FedGVI}{\acro{FedGVI}}

\allowdisplaybreaks

\def\P{{\mathcal{P}}}
\def\R{{\mathbb{R}}}

\definecolor{ibm0}{RGB}{255,176,0}
\definecolor{ibm1}{RGB}{254,97,0}
\definecolor{ibm2}{RGB}{220,38,127}
\definecolor{ibm3}{RGB}{120,94,240}
\definecolor{ibm4}{RGB}{100,143,255}

\definecolor{accessibleBlue}{RGB}{0,0,0}
\definecolor{accessibleVermillion}{RGB}{0,114,178}

\title[Rates of Convergence of GVI Posteriors under Prior Misspecification]{Rates of Convergence of Generalised Variational Inference Posteriors under Prior Misspecification}

\author[T. Mildner]{Terje Mildner}
\address{Department of Statistics and Department of Computer Science, University of Warwick, Coventry, UK}
\email{{Terje.Mildner@kellogg.ox.ac.uk;Terje.Mildner@gmail.com}}

\author[P. Giampouras]{Paris Giampouras}
\address{Department of Computer Science, University of Warwick, Coventry, UK}
\email{Paris.Giampouras@warwick.ac.uk}

\author[T. Damoulas]{Theodoros Damoulas}
\address{Department of Statistics and Department of Computer Science, University of Warwick, Coventry, UK}
\email{T.Damoulas@warwick.ac.uk}

\begin{document}
	\date\today

\begin{abstract}
We prove rates of convergence and robustness to prior misspecification within a Generalised Variational Inference (GVI) framework with bounded divergences. This addresses a significant open challenge for GVI and Federated GVI that employ a different divergence to the Kullback--Leibler under prior misspecification, operate within a subset of possible probability measures, and result in intractable posteriors.
Our theoretical contributions extend to misspecified priors that lead to inconsistent Bayes posteriors. 
In particular, we are able to establish sufficient conditions for existence and uniqueness of GVI posteriors on arbitrary Polish spaces, prove that the GVI posterior measure concentrates on a neighbourhood of loss minimisers, and extend this to rates of convergence regardless of the prior measure.
\end{abstract}
\maketitle

\section{Introduction}\label{sec:intro}

Consider the problem of posterior inference after having constructed a prior measure $\Pi$ and a hypothesis space $(\Theta,\mathcal{T})$ with associated hypothesis measures in the model $\P=\{P_\theta:\theta\in\Theta\}$. We are concerned with situations in which (a) the prior is misspecified, leading to inconsistent Bayes posteriors \cites{diaconis1986a,diaconis1986b}, (b) the model is misspecified, not containing the data generating process \cite{kleijn2006}, and (c) computational capabilities limit our ability to conduct posterior inference \cite{pinski2015}. 

To this effect we consider posterior inference through Generalised Variational Inference (GVI) as introduced by Knoblauch, Jewson, and Damoulas \cite{jeremias2022}. 
The present paper is concerned with characterising the dynamics of updating a posterior measure in the presence of (a), (b), and (c) through GVI and places particular emphasis on prior misspecification. In particular, we aim to answer the following open problems: under what conditions is a Generalised Variational Inference posterior measure robust to prior misspecification? Do GVI posteriors exist when observations come from an infinite dimensional stochastic process and are these asymptotically consistent? At what rate does this convergence take place?

Throughout, we consider $X_1,X_2,...$, taking values in $(\Xi,\mathcal{X})$, a sequence of observable random variables with unspecified dependencies generated by some unknown data generating process with $P_0$ its infinite dimensional process distribution. GVI recasts Bayesian updating as a problem in the calculus of variation, where we link observables and hypotheses through some loss function $L(X_1^n,\,\theta)$, regularise with respect to the prior through a divergence $D(\cdot:\Pi)$, and optimise over a restricted space of possible posterior measures $\Q\subset\PTheta$. The GVI posterior is defined as
\begin{equation}\label{eqn:gvi_min}
Q_n\in\argmin_{\nu\in\mathcal{Q}}\,\left\{T_n(\nu):=n\cdot J_{L_n}(\nu)+\beta ^{-1}D(\nu:\Pi)\right\}
\end{equation}
whenever it exists. Here $J_{L_n}(\nu):=\E_\nu[L(X_1^n,\,\theta)]$ is the linear functional given by integration of the loss with $n$ observables against some probability measure $\nu\in\Q$ and the constant $\beta>0$ denotes a learning rate parameter. As our results do not require there to be any data, we may simply drop this dependence whenever convenient and denote $L(X_1^n,\theta)=L_n(\theta)$.

The study of asymptotic consistency is classical for Bayes as it justifies an updating scheme over its frequentist counterpart by illustrating that we may come arbitrarily close to the optimal given sufficient observations \cites{schwartz1965,berk1966,diaconis1986a,barron1999,ghosal2000,shen2001,walker2004,ghosal2007,shalizi2009,walker2013}.
Considering now the question of asymptotic consistency of posterior measures reveals why it is advantageous to study GVI posteriors rather than their (Generalised) Bayesian counterparts under prior misspecification. Consider $(\Pi_n)_n$ the sequence of Bayes posteriors, and $d(\cdot,\cdot)$ a dissimilarity score between hypothesis measures, then the Bayes posteriors are consistent if for all $\varepsilon>0$ it holds with $P_0$ probability 1 that 
$$\Pi_n(P_\theta\in\P\cap\supp{\Pi}:d(P_\theta,P_0)>\varepsilon\,|\,X_1,X_2,...,X_n)\rightarrow0.$$
The choice of dissimilarity measure $d(\cdot,P_0)$ now illustrates our interest in GVI. A typical such measure is given as $d(P_\theta:P_0)=D_{KL}(P_\theta:P_0)-\essinf_{Q\in\P}^\Pi D_{KL}(Q:P_0)$ where the essential infimum over $\P$ is taken with respect to the prior $\Pi$, see for instance \cites{berk1966,walker2004,shalizi2009}. Further, we may only concentrate on measures in the support of the prior which is immediate in the formulation of the Bayes posterior or its variational counterpart in \cref{eqn:gvi_min} (see \cite{zellner1988}). %

The Bayesian, by nature of being Bayesian, can elicit arbitrary priors and may hence encounter prior misspecification; she must either suppose the prior encodes the best available knowledge of the underlying process---this however is scarcely practicable and processes exist for which we cannot intuit priors that lead to consistent posteriors---or accept the arbitrariness \cites{diaconis1986b,kleijn2006}. In fact, the nonparametric setting intrinsically encumbers the choice of reasonable prior as natural choices may lead to inconsistency \cites{diaconis1986a,diaconis1986b}, due to these not satisfying structural assumptions as required for consistency \cites{kleijn2006, shalizi2009}. Hence, the prior may be misspecified.

Moreover, it is scarcely the case that the model is well specified, that is the model family $\Theta$ will not contain the data generating process \cites{berk1966,shalizi2009,walker2013}. 
Furthermore, since the Kullback--Leibler divergence is known to not be robust (cf. \cite{jewson2018}), (b) may lead to suboptimal performance due to the choice of the negative log likelihood as a loss function.
Lastly, as Bayesian inference is rarely computationally tractable, optimising over a simpler space of probability measures---as done in Variational Inference with e.g. the set of Gaussians \cite{katsevich2024}---allows for computational feasibility. 

Generalised Variational Inference recasts Bayesian posterior updating as an optimisation problem in the calculus of variation over a simpler space of measures, allowing decision makers to mitigate (a), (b), and (c); see \cref{fig:misspec} for reference.

\begin{figure}
	\centering
	\begin{tikzpicture}[thick,scale=0.7, every node/.style = {scale=0.7}]%
		\draw[dashed,opacity=0.6] (-5,-3) -- (-2,0) -- (4,0);
		\draw[dashed,opacity=0.6] (-2,-0) -- (-2,6);

		\draw[opacity=0.15] (-4,4) -- (2,4) -- (2,-2) -- (-4,-2) -- (-4,4);
		\draw[opacity=0.15] (-3,5) -- (3,5) -- (3,-1) -- (-3,-1) -- (-3,5);
		\draw[opacity=0.15] (-5,-1) -- (1,-1)-- (4,2) -- (-2,2) -- (-5,-1);
		\draw[opacity=0.15] (-5,1) -- (1,1) -- (4,4) -- (-2,4) -- (-5,1);
		\draw[opacity=0.15] (-3,-3)--(-3,3)--(0,6 )-- (0,0)--(-3,-3);
		\draw[opacity=0.15] (-1,-3)--(-1,3)-- (2,6) -- (2,0) -- (-1,-3);

		\draw[opacity=0.15] (-4,0) -- (2,0);
		\draw[opacity=0.15] (-3,1) -- (3,1);
		\draw[opacity=0.15] (-4,2) -- (2,2);
		\draw[opacity=0.15] (-3,3) -- (3,3);
		
		\draw[opacity=0.15] (-2,-2) -- (-2,4);
		\draw[opacity=0.15] (0,-2) -- (0,4);
		\draw[opacity=0.15] (-1,-1) -- (-1,5);
		\draw[opacity=0.15] (1,-1) -- (1,5);
		
		\draw[opacity=0.15] (-3,-1) -- (0,2);
		\draw[opacity=0.15] (-1,-1) -- (2,2);
		\draw[opacity=0.15] (-3,1) -- (0,4);
		\draw[opacity=0.15] (-1,1) -- (2,4);

		\fill[fill=white!100,opacity=0.8] (-3.9,-1.6) -- (-3.1,-1.6) -- (-3.1,-1.4) -- (-3.9,-1.4) -- (-3.9,-1.6);
		\fill[fill=white!100,opacity=0.8] (-2.7,-0.3) -- (-1.3,-0.3) -- (-1.3,0.3) -- (-2.7,0.3) -- (-2.7,-0.3);
		\fill[fill=white!100,opacity=0.8] (-4.25,2.05) -- (-2.7,2.05) -- (-2.7,1.5) -- (-2.85,1.5) -- (-2.85,0.95) -- (-4.15,0.95) -- (-4.25,2.05);
		\fill[fill=white!100,opacity=0.8] (1.9,1.8) -- (1.1,1.8) -- (1.1,1.2) -- (1.9,1.2) -- (1.9,1.8);
		
		\fill[fill=ibm2!30,opacity=0.5] (-2,-2) -- (-4,-2) -- (-4,0) -- (-2,2) -- (0,2) -- (0,0) -- (-2,-2);
		\node[align=center] at (-2,0) {Variational\\Inference};
		
		\fill[fill=ibm0!30,opacity=0.5] (-5,-3) -- (-5,-1) -- (-4,0) -- (-2,0) -- (-2,-2) -- (-3,-3) --(-5,-3) ;
		\node[align=center] at (-3.5,-1.5) {Bayes};
		
		\fill[fill=ibm4!30,opacity=0.5] (-5,-1) -- (-5,3) -- (-4,4) -- (-2,4) -- (-2,0) -- (-3,-1) -- (-5,-1);
		\node[align=center] at (-3.5,1.5) {Generalised\\Bayesian\\Inference};
		
		\fill[fill=ibm3!30,opacity=0.5] (-1,-3) -- (-1,3) -- (2,6) -- (4,6) -- (4,0) -- (1,-3) -- (-1,-3);
		\node[align=center] at (1.5,1.5) {Our\\Work};

		\node[align=center,color=accessibleBlue] at (-4,-3.35) {None};
		\node[align=center,color=accessibleBlue] at (-2,-3.35) {Mild};
		\node[align=center,color=accessibleBlue] at (0,-3.35) {Severe};
		
		\node[align=center,color=accessibleBlue] at (-2,-4) {Prior Misspecification};
		\draw[->|] (-3.75,-4) -- (-5,-4);
		\draw[->|] (-0.25,-4) -- (1,-4);
		
		\node[align=center,color=accessibleBlue] at (-6.4,-2) {Well Specified\\Model};
		\node[align=center,accessibleBlue] at (-6.4,0) {Mild Model\\Misspecification};
		\node[align=center,accessibleBlue] at (-6.4,2) {Severe Model\\Misspecification};

		\node[align=center,accessibleBlue] at (-5.75,3.75) {No Computational\\Constraints};
		\node[align=center,color=accessibleBlue] at (-4.85,4.75) {Mild Computational\\Constraints};
		\node[align=center,color=accessibleBlue] at (-3.95,5.75) {Severe Computational\\Constraints};

		\node[align=center,color=accessibleVermillion,font=\itshape] at (-1,6.5) {KL\\Divergence};
		\node[align=center,color=accessibleVermillion,font=\itshape] at (1,6.5) {Unbounded\\Divergence};
		\node[align=center,color=accessibleVermillion,font=\itshape] at (3,6.5) {Bounded\\Divergence};
		
		\node[align=center,color=accessibleVermillion,font=\itshape] at (5.35,1) {Negative\\Log--Likelihood};
		\node[align=center,color=accessibleVermillion,font=\itshape] at (5.35,3) {Divergence--\\Based Loss};
		\node[align=center,color=accessibleVermillion,font=\itshape] at (5.35,5) {Arbitrary\\Loss};
		
		\node[align=center,color=accessibleVermillion,font=\itshape] at (4.75,-0.75) {Parametrised\\$\mathcal{Q}\subset\PTheta$};
		\node[align=center,color=accessibleVermillion,font=\itshape] at (3.5,-1.75) {Arbitrary\\$\mathcal{Q}\subset\PTheta$};
		\node[align=center,color=accessibleVermillion,font=\itshape] at (2.2,-2.75) {$\PTheta$};

		\draw[] (-5,-3) -- (-5,3) -- (1,3) --(1,-3) -- (-5,-3) -- (-5,3);
		\draw[] (-5,3) -- (-2,6) -- (4,6) -- (1,3);
		\draw (4,1) -- (4,0) -- (1,-3) -- (0,-3);
		\draw (3,-1) -- (4,0) -- (4,6) -- (3,6);
	\end{tikzpicture}
\caption{A taxonomy of Generalised Variational Inference.}
\label{fig:misspec}
\end{figure}
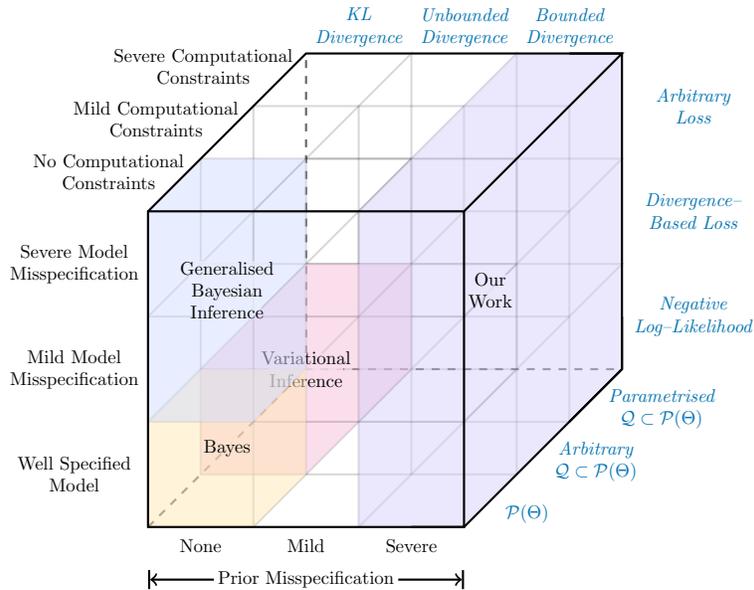
Notably, our work differs significantly from Bayesian sensitivity analysis, and robustness to the prior of a Bayesian posterior, as discussed for instance in \cites{berger1985,gustafson1995}. We are not concerned with small perturbations of the prior towards some other probability measure, but rather with priors that will in general not result in consistent Bayes posteriors, such as those that are mutually singular with the GVI posterior, or not contain KL neighbourhoods of the data generating process in their support. We provide a fuller account in the appendix.

In this paper, we develop theory concerning GVI posterior measures over infinite dimensional hypothesis spaces $\Theta$. The main contribution of this paper is \cref{thm:rates} where we establish rates of convergence of GVI posterior measures just slower that $n^{-1}$ when using bounded divergences. In doing so we first present sufficient conditions for existence and uniqueness of GVI posteriors, which we then use as a basis to characterise asymptotic concentration to sets containing the minimisers of the loss.
Significantly, our theory allows us to deal with instances in which the prior is severely misspecified, such as in the case where the prior has full measure on a subset of $\Theta$ where the Kullback--Leibler divergence rate is infinite and where the prior and GVI posterior are mutually singular. This is highlighted in \cref{fig:misspec}.

This paper is structured as follows. In \cref{sec:background} we provide the necessary background for Generalised Variational Inference, prior misspecification, and the assumptions used throughout the paper. 
\cref{sec:results} states and discusses the main results of this paper, \cref{sec:extensions} presents extensions to these, including predictive consistency of GVI under standard losses, and we provide the complete proofs in \cref{sec:proofs}. 
To highlight our results, \cref{sec:examples} illustrates how GVI is robust to prior misspecification in cases where Bayesian posteriors are inconsistent. 
In \cref{sec:discussion} we provide some concluding remarks on our results, their relevance to extensions of GVI in Federated Learning \cite{mildner2025}, and relevant open problems.

\section{Preliminaries}\label{sec:background}

\subsection{Notation} 

Let $X_1, X_2, ..., X_n$ be a sequence of observable random variables, short $X_1^n$, generated by the abstract probability space $(\Omega, \mathcal{F}, P_0)$ and taking values in the measurable space $(\Xi, \mathcal{X})$ with $P_0$ distribution.
We consider $n\in\mathbb{N}$ and let this tend to infinity for the concentration results. Denote by $\sigma(X_1^n)$ the natural filtration of this sequence. 
We do not place any restrictions on the dependence of the data, as long as the loss associated with it satisfies the respective assumptions, especially \cref{asp:L-ae}.

We define our hypothesis measures $P_\theta$, with $\theta$ taking values in the measurable space, $(\Theta, \mathcal{T})$ where $\Theta$ is a Polish space, in general infinite dimensional, endowed with its Borel $\sigma$--algebra, and we assume that $P_0$ and $P_\theta$ admit densities $p_0$ and $p_\theta$ respectively%
. As we allow for model misspecification, in the nonparametric setting we suppose that $P_0\notin\Theta$ in general\footnote{In the parametric setting this translates to $\nexists\theta_0\in\Theta$ such that $P_{\theta_0}=P_\theta$ in the sense of measures.}. 
However, we may find elements of $\Theta$ that are closest, in some sense, to the data generating process $P_0$, which we will denote using the `$^\star$' notation.

Throughout, we will study probability measures in $\PTheta,$ the space of all Borel probability measures on $(\Theta, \mathcal{T})$. Note that $\Theta$ being Polish implies that $\PTheta$ is also Polish, and hence allowing for arbitrary Polish spaces allows for nonparametric inference \cites{kechris1995,dembo2010}, we identify $\Theta$ and $\P$. We begin with some prior $\Pi$ and update this to the GVI posterior $Q_n$.

We further define the terms of \cref{eqn:gvi_min}. Let $D:\PTheta^2\rightarrow\mathbb{R}_{\ge0}$ be a statistical divergence between probability measures. $L(X_1^n,\,\theta)$ is a loss function connecting the data and the hypotheses, and $\mathcal{Q}\subset\PTheta$ is a (sub)space of probability measures we are optimising over. Clearly $n\in\mathbb{N}$ is the number of observables, and the term $\beta>0$ represents a fixed learning rate parameter. We later relax the assumptions on $\beta$.

We say that eventually $a_n\gtrsim b_n$ to denote the existence of an unspecified constant $k>0$, depending not on $n$, such that for all $n$ sufficiently large it holds that $a_n\ge k\cdot b_n$. For a measure $\nu\in\PTheta$ and a set $A\in\mathcal{T}$ we denote $\nu A:= \nu(\{\theta\in\Theta:\theta\in A\})$. Note when using $\subset$, we do not distinguish between strict containment or up to equality as it bears no significance on the theory. We denote $\mathbb{R}_+:=\{a\in\mathbb{R}: a>0\}$ and $\mathbb{R}_{\ge 0}:=\{a\in\mathbb{R}: a\ge0\}$.

\subsection{Assumptions}\label{sec:asp}
The assumptions required throughout the paper are primarily related to the behaviour of the data generating process $P_0$ and its interaction with the choices of loss functions, divergence, and the set of measures $\Q$.
\begin{assumption}\label{asp:loss}
The loss $L(X_1^n,\,\theta)$ is $\sigma(X_1^n)\times\mathcal{T}$--measurable and lower bounded\footnote{Loss functions used in ML are typically non--negative, we, in accordance with the literature on GVI and GBI, allow for more general functions through the lower boundedness.} 
for all $n\in\N$.
\end{assumption}

The next assumption requires for $\Q$ to be a sensible choice which does not make the objective pathological.
\begin{assumption}\label{asp:finite}
	There exists at least one $Q\in\mathcal{Q}$ such that $T_n(Q)<\infty$.
\end{assumption}

In order to reason about the dynamical properties of GVI posteriors we require that minimisers of \cref{eqn:gvi_min} exist. In particular, we impose the following two assumptions which will be sufficient to prove existence of GVI posteriors in \cref{thm:existenceGVI} given that $\Q$ is weak$^\star$ closed.
\begin{assumption}\label{asp:lscL}
	The losses as functions on $\Theta$, $L_n:\Theta\rightarrow\mathbb{R}\cup\{+\infty\}$, are coercive and lower semi--continuous for all $n\in\mathbb{N}$.
\end{assumption}
This assumption is slightly stronger than required as we require for the functional $\nu\mapsto J_{L_n}(\nu)$ to be lower semi--continuous; the loss having this property immediately ensures this. As a consequence however, the losses will achieve their infima on any compact set. Moreover, we are able to discard the assumption of the coerciveness of the loss whenever $\Theta$ is compact.

In order for the entire GVI functional $\nu\mapsto T_n(\nu)$ to be lower semi--continuous, we require the same to hold for the divergence.
\begin{assumption}\label{asp:lscD}
	The divergence $D:\PTheta\times\PTheta\rightarrow\mathbb{R}_{\ge0}\cup\{+\infty\}$ is lower semi--continuous in its first argument.
\end{assumption}
We remark here that this assumption is satisfied by many common choices of divergence including the Kullback--Leibler and the Total Variation.
This, together with \cref{asp:lscL}, provides a key tool in proving the existence and uniqueness of GVI posterior measures in \cref{thm:existenceGVI}.

To characterise the robustness of GVI to prior misspecification, we analyse the behaviour of GVI posteriors under different priors. In particular, we denote by $\mathcal{G}\subset\PTheta$ the set of possible priors. It turns out that the following assumption provides a sufficient condition for consistency, and in fact rates of convergence, of GVI to hold regardless of the prior $\Pi\in\mathcal{G}$.

\begin{assumption}\label{asp:bdd} 
For given $\mathcal{Q}$ and $\mathcal{G}$ the divergence in \cref{eqn:gvi_min} is bounded above, i.e. $\exists M> 0$ such that $\forall Q\in\mathcal{Q}$ and $\forall\Pi\in\mathcal{G}$ we have $D(Q:\Pi)\le M$.
\end{assumption}

\begin{remark}
It is worth noting that this may be satisfied by both bounded divergences, such as the total variation distance, as well as specific choices of the set of measures we are optimising over, $\mathcal{Q}$, and the set of possible priors, $\mathcal{G}$. In the former case, we may allow for $\mathcal{G}=\PTheta$ and $\Q\subset\PTheta$ without restriction. The latter is more restrictive and for instance includes the case where $\mathcal{G}$ is a singleton $\{\Pi\}$ and $\Q\subset\{\nu\in\PTheta:D(\nu:\Pi)\le M\}$.
\end{remark}

Finally, we require that the dynamics of the data generating process $P_0$ and the loss functions, $L_n$, mesh sufficiently well together.
\begin{assumption}\label{asp:L-ae}
There exists a $\mathcal{T}-$measurable function $L:\Theta\rightarrow\mathbb{R}$ such that $\forall\theta\in\Theta$
$$
\lim_{n\rightarrow\infty}L(X_1^n,\,\theta)=L(\theta)
$$
$P_0$ almost--surely, which is lower semi--continuous and that it $P_0$--a.s. holds that
$$\inf_{\theta\in\Theta}L(\theta)\ge\limsup_n\inf_{\theta\in\Theta}L(X_1^n,\theta)>-\infty\,.$$
\end{assumption}
This assumption is required for the concentration results of \cref{thm:asymptotics,thm:rates}, and the first part is standard in Bayesian asymptotics (see e.g. \cite{shalizi2009}) where an asymptotic equipartition property, as in the Shannon--MacMillan--Breimann theorem, is required to hold. 
Unless otherwise stated all convergence takes place in the number of observables $n$.

As we will be working with the infima of the different loss functions throughout, we follow \cite{shalizi2009} by adopting  further notational abbreviations stemming from large deviation theory.
\begin{gather}\label{eqn:infDef}
L_n(A)\equiv\inf_{\theta\in A}\,L(X_1^n,\,\theta)\,,\qquad
L(A)\equiv\inf_{\theta\in A}\,L(\theta)
\end{gather}
As we have not fixed a metric, we follow \cites{berk1966,shalizi2009,miller2021} in using:
\begin{equation}
N_\varepsilon:=\{\theta\in\Theta:L(\theta)\le L(\Theta)+\varepsilon\}\,.
\end{equation}

\subsection{Misspecification of the Model and the Prior}\label{sec:prior_mis}
As identified in \cites{walker2013,shalizi2009} the elements of $\Theta$ that the Bayesian posterior targets are those minimising the Kullback--Leibler divergence between the data generating process and the hypothesis measure. Denoting this as a function on $\Theta$, we write:
$$
h(\theta)=\lim_{n\rightarrow\infty}\frac1n\mathbb{E}_{P_0}\left[\log \frac{p_0(x_1^n)}{p_\theta(x_1^n)}\right]
$$
which may possibly be infinite; note that the expectation is with respect to $P_0$ and hence over $\Omega$. This is the loss function targeted in classical Bayesian inference, in particular the elements of $\Theta$ that minimise the above minimise the Kullback--Leibler divergence between the data generating process and the hypothesis measure. However, this may be suboptimal in many cases where the model is misspecified, for instance not containing the data generating process, (see e.g. \cites{bernardo2000,mclatchie2025}). It may therefore be desirable to instead minimise a different robust divergence, such as a density power divergence \cite{basu1998}, but we do not necessitate the loss to approximate one and it may be more general.

Moreover, if the prior has full measure on the set of elements where this KL rate is infinite, that is on the set $I:=\{\theta\in\Theta:h(\theta)=\infty\}$, then all previous results on consistency, let alone rates of convergence, of (Generalised) Bayesian inference fail \cite{shalizi2009}. In this case we say that the prior is severely misspecified. In this we also include priors that lead to inconsistent posteriors such as those observed in \cite{diaconis1986b}.
To make this rigorous, we place the following definition of robustness to prior misspecification.
\begin{definition}\label{def:asymPR}
Let $\mathcal{G}\subset\PTheta$ be the set of all possible priors. We say a posterior is asymptotically robust to prior misspecification w.r.t. $\mathcal{G}$ if the posterior is consistent for all priors $\Pi\in\mathcal{G}$. %
\end{definition} 
Classical results on Bayesian sensitivity analysis deal with small perturbations of some reference prior by some other measure through additive or multiplicative misspecification (see \cites{dey1994, gustafson1995}). That is some reference prior $\Pi$ is perturbed by another probability measure $G\in\PTheta$ through a constant $\varepsilon\in[0,1]$, e.g.
$\Pi_\varepsilon=(1-\varepsilon)\Pi+\varepsilon G$. Our results are distinct to this line of work as we deal with asymptotic inconsistency, and we will not use such explicit formulations for $\mathcal{G}$. We discuss this distinction in \cref{apx:classicPriorRobust}.

\subsection{Generalised Variational Inference}
The Bayesian posterior---consisting of the likelihood function of the data, $p_\theta(X_1^n)$, and a prior measure, $\Pi$---is the random measure
$$\Pi_nA:=\cfrac{\int_A p_\theta(X_1^n)\,\Pi(d\theta)}{\int_\Theta p_\theta(X_1^n)\,\Pi(d\theta)}= \cfrac{\int_A R_n(\theta)\,\Pi(d\theta)}{\int_\Theta R_n(\theta)\,\Pi(d\theta)}$$
as usual. It is immediate that we may replace $p_\theta(X_1^n)$ by $R_n(\theta):=p_\theta(X_1^n)/p_0(X_1^n)$ without changing $\Pi_n$ to see why we consider $h(\theta)$ above. In particular, under mild conditions, the generalised Shannon--MacMillan--Breiman theorem \cite{algoet1988} guarantees that $n^{-1}\log R_n(\theta)\rightarrow-h(\theta)$ (a.s.).
If we instead replace the likelihood, $p_\theta(X_1^n)$, by the exponential of a negative loss acting on the data, that is by $\exp\{-nL(X_1^n,\theta)\}$, we get Generalised Bayesian Inference (GBI) \cites{bissiri2016,miller2021}, or a Gibbs posterior \cite{alquier2016}. If we discard the prior, we may instead get an energy--based, frequentist model. 

As shown in \cite{zellner1988} Bayes may be equivalently formulated as an optimisation problem where $\Pi_n$ is the measure minimising the `Evidence Lower Bound'---the expectation of the negative log--likelihood, and the Kullback--Leibler divergence to the prior---over all probability measures acting on the hypothesis space. 
Tangentially, minimising a divergence other than the KL divergence to the Bayesian posterior may result in nicer optimisation problems in practice, %
this however does not nicely decompose into an expectation acting on the loss and a divergence measure acting on the prior, conflating model and prior misspecification; see also section 2.3.2 in \cite{jeremias2022}. GVI instead replaces the negative log--likelihood by any loss function, similar to GBI, as well as the KL divergence by any divergence, generalising these previous approaches. Typical choices of loss functions are those that either offer attractive computational benefits, and those that incorporate the hypothesis measure robustly, including those based on robust divergences; see \cite{mildner2025}.

Having introduced the necessary assumptions in \cref{sec:asp} we are able to establish that GVI posteriors over arbitrary Polish spaces minimising $\nu\mapsto T_n(\nu)$ exist and, under further assumptions, are also unique, extending the previous results of Knoblauch \cite{knoblauch2021} and Wild et al. \cite{wild2023}. In particular, we allow for $\Theta$ to be an infinite dimensional Polish space, rather than a finite dimensional, complete, and separable normed space. Far less do we require for $\mathcal{Q}$ to be convex; this is only required to establish uniqueness.

\begin{theorem}\label{thm:existenceGVI}
Make \cref{asp:loss,asp:lscL,asp:lscD,asp:finite}. Suppose also that $\mathcal{Q}$ is closed with respect to weak$^\star$ convergence. Then a GVI posterior minimising \cref{eqn:gvi_min} exists. 

If it further holds that the divergence is strictly convex in its first argument and that $\mathcal{Q}$ is convex, then there exists a unique minimiser to the functional $\nu\mapsto T_n(\nu)$.
\end{theorem}
Establishing the existence of GVI posterior measures provides that the GVI objective is in fact well specified and a reasonable framework for decision makers in the nonparametric setting.

\section{Main Results}\label{sec:results}
By \cref{thm:existenceGVI} we know that the sequence of GVI posteriors exist, so we may study their asymptotic behaviour. In key difference to previous work on (generalised) Bayesian posteriors \cites{miller2021, kleijn2006, shalizi2009} we do not require explicit knowledge of the GVI posterior. Furthermore, when comparing our results to the frequentist consistency result of Knoblauch \cite{knoblauch2021}, the assumptions on the loss (and the data) are far less strict, and we do not require $\Theta$ to be a finite dimensional separable normed space. While the bounded divergence is a stricter requirement this additionally allows for robustness to prior misspecification as in \cref{def:asymPR} and, given the right choice of divergence, hold for all priors $\Pi\in\PTheta$. 

To reason about the asymptotics of the sequence of GVI posterior measures we need that these exist, so we suppose in the remaining that $\mathcal{Q}$ is closed with respect to weak$^\star$ convergence. 
A key challenge here is that we require for the family $\mathcal{Q}$ to be rich enough, that is the measures that we desire to concentrate on exist within $\mathcal{Q}$ a priori. Naturally, this challenge is non--existent in consistency results of Bayesian posteriors as $\mathcal{Q}=\PTheta$, we however need to be able to deal with subsets of $\PTheta$. 
We solve this issue by requiring for $\mathcal{Q}$ to allow for Dirac measures; by this we mean that $\mathcal{Q}$ is closed with respect to the weak$^\star$ convergence of probability measures, and that for all $n$ sufficiently large it holds $P_0$--a.s. that $\inf_{\nu\in\mathcal{Q}}\,J_{L_n}(\nu)=L_n(\Theta)$. Recall that we defined $L_n(A)\equiv\inf_{\theta\in A} L_n(\theta)$ in \cref{eqn:infDef}.
\begin{theorem}\label{thm:asymptotics}
Make \cref{asp:loss,asp:finite,asp:bdd,asp:lscL,asp:lscD,asp:L-ae}. Suppose further that $\mathcal{Q}$ allows for Dirac delta measures, and that $\forall A\in\mathcal{T}$ for which $L(A)>L(\Theta)$ we almost surely have that
$$
\liminf_nL_n(A)>L(\Theta)\,.
$$
Then for any such $A\in\mathcal{T}$ it $P_0$--a.s. holds that $Q_nA\rightarrow0$.

Moreover, suppose there exists a unique $\theta^\star\in\Theta$ minimising $L(\theta)$, then $$Q_n\overset{{w}^\star}{\longrightarrow}\delta_{\theta^\star}\,$$ $P_0$--a.s.
\end{theorem}

This result implies concentration of the GVI posterior on a subset of $\Theta$ on which the converged loss $L$ is minimised.  In fact we may pick $A$ such that we concentrate on its complement, in particular $A^c=N_\varepsilon:=\{\theta:L(\theta)\le L(\Theta)+\varepsilon\}$ for any $\varepsilon>0$.

In the moreover part we show the weak$^\star$ ($w^\star$) convergence of $(Q_n)$ to $\delta_{\theta^\star}$, which we may interpret as the expected loss minimiser of $L(\theta)$, i.e. $P^\star_\infty$. The extension to infinite dimensional Polish spaces and the weakening of assumptions make this is a considerable improvement over the previous result on weak$^\star$ convergence in \cite{knoblauch2021}.

As our main result of the paper, we are able to derive a rate of convergence result in the number of observations $n$. In fact, we may pick a non--increasing sequence $\varepsilon_n$ decaying slowly enough such that the concentration still holds on sets $N_{\varepsilon_n}:=\{\theta\in\Theta:L(\theta)\le L(\Theta)+\varepsilon_n\}$.
\begin{theorem}\label{thm:rates}
Make \cref{asp:loss,asp:finite,asp:bdd,asp:lscL,asp:lscD,asp:L-ae}, and suppose $\mathcal{Q}$ allows for Dirac delta measures. 
Pick a positive sequence $\varepsilon_n$ such that $\varepsilon_n\rightarrow0$ and $n\varepsilon_n\rightarrow \infty$. 
If  $$L_n(N_{\varepsilon_n}^c)-L_n(\Theta)\gtrsim\varepsilon_n$$
eventually almost surely, then
\begin{equation}
	Q_nN_{\varepsilon_n}\rightarrow 1  %
\end{equation}
$P_0$--a.s.
\end{theorem}

Significantly, this is the first result that establishes rates of convergence for Generalised Variational Inference posterior measures. 
Here we allow for sequences $\varepsilon_n$ that converge just slower than $n^{-1}$ which should be compared to the result of \cite{shalizi2009} who finds the same rates for Bayesian posteriors with highly correlated data using Kullback--Leibler neighbourhoods. In \cite{ghosal2007} the authors finds rates of order $n^{-1/2}$ for the Bayesian posterior using stronger definitions of $N_{\varepsilon_n}$ through the Hellinger metric.

Moreover, given that \cref{asp:bdd} applies, we find that the above theorems are independent of the choice of prior $\Pi\in\mathcal{G}$. This implies that we do not require $\Pi(N_{\varepsilon_n})>0$ for any $n\in\mathbb{N}$, given $\mathcal{G}=\PTheta$, and the GVI posterior may hence concentrate on hypotheses that minimise the converged loss, $L$, rather than being restricted to the support of the prior which may be misspecified.

There is no objective reason why the different inputs into the GVI objective should be fixed with respect to incoming observations. In fact, \cref{thm:asymptotics,thm:rates} remain valid even if $\beta$, $M$, $\Pi$, $L$, or $D$ vary with $n$. We simply require that the respective assumptions are satisfied for each $n$, but these may affect the rate of convergence.

\begin{theorem}\label{thm:n}
	\cref{thm:asymptotics,thm:rates} remain valid if $\Pi$, $\Theta$, $\P$, $\mathcal{Q}$, $L(\cdot,\cdot)$, $M$, $D(\cdot:\cdot)$, and $\beta$ depend on $n$ but satisfy the respective assumptions for each $n$. In particular, we require $M_n\beta_n^{-1}n^{-1}\rightarrow 0$, and that the sequence $\varepsilon_n$ satisfies $n\varepsilon_nM_n^{-1}\beta_n\rightarrow\infty$.
	
	Moreover, \cref{thm:asymptotics,thm:rates} hold regardless of the prior $\Pi\in\mathcal{G}$, and we may pick the divergence $D$ such that $\mathcal{G}=\PTheta$.
\end{theorem}

\begin{remark}
Notably, this allows for situations in which $Q_n$ and $\Pi$ are mutually singular, and hence whenever the prior is severely misspecified.
\end{remark}
As another extension of the preceding theorems, it is certainly permissible for the prior to put mass outside the space $\Theta$. In fact, we may even allow for its entire mass to be outside of $\Theta$, in which case the divergence is maximal. Such situations may arise when the space $\Theta$ is embedded inside some larger space $T$. Consider the case where $\Theta$ corresponds to some model class $\mathcal{P}\subset\mathcal{P}(\mathbb{R})$ and the prior is in fact over the entire space $\mathcal{P}(\mathbb{R})$.

\section{Finite Sample Bound and Extensions}\label{sec:extensions}
\subsection{A Finite Sample Bound}
Before we bound the objective of the GVI posterior under arbitrary priors $\Pi\in\mathcal{G}$, we define an empirical loss minimising measure corresponding to a minimiser of the functional $\nu\mapsto J_{L_n}(\nu)$.
\begin{definition}\label{asp:mle}
	We denote $P_n^\star\in\mathcal{Q}$ such that $\forall Q\in\mathcal{Q}$ it holds that $J_{L_n}(P_n^\star)\le J_{L_n}(Q)$. This is guaranteed to exist under \cref{asp:loss,asp:finite,asp:lscL} supposing that $\Q$ is weak$^\star$ closed.
\end{definition}
\begin{proposition}\label{thm:set}
	Under \cref{asp:loss,asp:finite,asp:bdd,asp:lscL}, for all priors $\Pi\in\mathcal{G}$, the respective GVI posteriors $Q_n$ of \cref{eqn:gvi_min} are in the set
	\begin{equation*}
		\mathcal{R}_n^\star:=\left\{
		Q\in\mathcal{Q}: \int_\Theta L(X_1^n,\,\theta)\left[Q(d\theta)-P_n^\star(d\theta)\right]\le \frac M{n\beta}
		\right\}.
	\end{equation*}
	Moreover, suppose that $\mathcal{Q}$ is closed with respect to weak$^\star$ convergence, then $\mathcal{R}^\star_n$ is compact with respect to weak$^\star$ convergence.
\end{proposition}
This result provides an instrumental role in proving the theorems of the preceding section.

\begin{remark}
\cref{asp:bdd} can also be satisfied by restricting the set of possible posterior measures, $\mathcal{Q}$, to the sublevel sets of an unbounded divergence to some prior:
$$\mathcal{Q}|_M:=\{\nu\in\mathcal{Q}:D(\nu:\Pi)\le M\}.$$
\cref{thm:set} may now be satisfied with $\mathcal{Q}|_M$ and $\mathcal{G}=\{\Pi\}$. This, however, does not guarantee robustness to prior misspecification in the same generality as using a bounded divergence, due to the set $\mathcal{G}$ from \cref{def:asymPR} being smaller. In fact, $\mathcal{Q}|_M$ implicitly makes assumptions on the absolute continuity of its members to the prior.
\end{remark}

\subsection{Generalisation Performance}
Let $P^n_0\equiv P_0(X_n|\sigma(X_1^{n-1}))$ be the conditional under the data generating process, and denote the posterior predictive $Q^n_{P_\theta}\equiv\int_\Theta P_\theta(X_n|\sigma(X_1^{n-1}))Q_n(d\theta)$ under the GVI posterior $Q_n$. It follows from \cite{shalizi2009} that the following result on the generalisation performance of GVI holds.
\begin{theorem}\label{thm:gen}
Let $L(X_1^n,\,\theta):=n^{-1} \log (p_0(X_1^n)/p_\theta(X_1^n))$, and suppose that the assumptions of \cref{thm:asymptotics} hold. Then it $P_0$--a.s. holds that
\begin{align*}
\limsup_{n\rightarrow\infty} d_H^2(P_0^n,Q^n_{P_\theta})&\le  L(\Theta) \\
\limsup_{n\rightarrow\infty} d_{tv}^2(P_0^n,Q^n_{P_\theta})&\le 4L(\Theta)
\end{align*}
where $d_H$ and $d_{tv}$ are respectively the Hellinger and total variation distances.
\end{theorem}
Note that this result holds independently of the choice of prior measure $\Pi\in\PTheta$, and $L(\Theta)$ corresponds to $h(\Theta)\equiv\inf_{\theta\in\Theta}h(\theta)$ from \cref{sec:prior_mis}.
\subsection{On Consistency of GVI Posteriors under Arbitrary Divergences}
For completeness, we also include a result on asymptotic consistency of infinite dimensional GVI posteriors under unbounded divergences. Unfortunately, this requires restrictive assumptions on the set $\mathcal{Q}$, and requires that the loss and the divergence are well behaved. Specifically, we require that a sequence of measures exists which concentrate around loss minimisers, and that eventually the contribution of the divergence of these measures is of order $\mathcal{O}(n^\alpha)$ for $\alpha<1$. Nevertheless, if this is satisfied we can conclude that the sequence of GVI posteriors concentrates on sets containing loss minimisers.
\begin{theorem}\label{thm:unboundedConcentration}
Make \cref{asp:loss,asp:finite,asp:lscL,asp:lscD,asp:L-ae}. Suppose also that there exists a sequence $(\nu_n)\subset\mathcal{Q}$, not necessarily GVI minimisers, for which (i) $(J_{L_n}(\nu_n)-L_n(\Theta))\rightarrow 0$ $P_0$--a.s., (ii) $n^{-1}D(\nu_n:\Pi)\rightarrow0$, and (iii) $\forall A\in\mathcal{T}$ for which $L(A)>L(\Theta)$ it almost surely holds that $\liminf_nL_n(A)>L(\Theta)$. Then, for all such sets $A$, $Q_nA\rightarrow0$ with $P_0$--probability 1.
\end{theorem}
Here, in key difference to \cref{thm:asymptotics}, we require an explicit choice of prior $\Pi$, and while this result allows us to reason about the concentration of arbitrary GVI posteriors on infinite dimensional Polish spaces $\Theta$, it unfortunately tells us nothing about the rate of convergence.

\section{Illustrative Example}\label{sec:examples}

\subsection{Mitigating Inconsistent Bayes Posteriors}
\cref{thm:asymptotics,thm:n} suggest that consistency holds independent of the choice of prior. Hence, we may mitigate inconsistency of Bayes posteriors through GVI with bounded divergences. To illustrate this, we follow the example of Diaconis and Freedman \cite{diaconis1986b} in modelling a location parameter $\theta_0$ given independent errors $\v_i$. We refer to their model as DF.
\begin{equation}\label{expl:DF_model}
X_i= \theta_0 +\v_i
\end{equation}
for some $\v_i$ independent with unknown distribution $F_0$. In particular, 
they suppose that it has some density $h_0\in C_c^\infty(\mathbb{R})$ symmetric about 0 which is a strict global maximum. DF model the prior for $\theta_0$ and $F_0$ independent, $\theta\sim\mathcal{N}(0,1)$ and $F\sim D(\alpha)$ the Dirichlet process prior based on the standard Cauchy. 

Notably, DF prove that there exist $h_0$, symmetric about $\theta_0$, for which it $P_0$--a.s. holds that the Bayes posterior does not concentrate on $\theta_0=0$ but rather around $\pm\gamma$ for $\gamma>0$ with equal probability. In particular, DF demonstrate $\forall\eta>0$ it holds that
$\limsup \Pi_n\{(\theta,F): |\theta-\gamma|<\eta\}=1\, a.e.$ and $\limsup \Pi_n\{(\theta,F): |\theta+\gamma|<\eta\}=1\, a.e.$ That is the posterior oscillates around $+\gamma$ and $-\gamma$. In fact, DF show that this holds even if we forgo identifiability issues associated with $\theta$ and $F$ by considering only symmetric densities $G$ given a symmetrised Dirichlet prior.

For our purposes, we consider the location parameter $\theta\in\R$ to have corresponding Dirac measure $\delta_\theta$, and consider $G$ to be in the set of symmetric probability measures. That is $G\in\overline{\P}:=\{\frac12(\nu+\nu^-):\nu\in\P(\R), \nu^-:=\nu(-dx)\}$, which are mean zero, symmetric measures. We consider the set of models to be their convolution $\delta_\theta\star G$, that is $\{\delta_\theta\star G:\theta\in\R,G\in\overline{\P}\}$ with $\Theta=\R\times\overline{\P}$.

We now demonstrate in \cref{prop:DF} that GVI updating with bounded divergences results in consistent posteriors.
\begin{proposition}\label{prop:DF}
Suppose the model is given as \cref{expl:DF_model} with $\v_i$ independent, $h_0\in C^\infty_c$ symmetric about zero and such that the Bayes posterior is inconsistent, and $\theta_0=0$ as in \cite[Theorem 3]{diaconis1986b}.
Let $\pi(d\theta,dG)$ be as above, $L(X_1^n,\,(\theta,G))$ the average negative log likelihood, and $\mathcal{Q}=\PTheta$ for $\Theta=\mathbb{R}\times\overline{\P}$. Then, under any divergence satisfying \cref{asp:lscD,asp:bdd}, the GVI posteriors will be consistent.

Moreover, it is $P_0$--a.s. that $\forall\eta>0$
$$\lim_{n\rightarrow\infty} Q_n\{\theta: |\theta|<\eta\}=1 \,\,a.e.$$
\end{proposition}
We give explicit examples of divergences satisfying both \cref{asp:lscD,asp:bdd} in \cref{sec:divs}. Note also that there exists an equivalent $N_\varepsilon$ to $N_\eta:=\{\theta\in\mathbb{R}:|\theta|<\eta\}$.
\begin{proof}
We verify the assumptions of \cref{thm:asymptotics}. By \cite[Thm. 17.23, Prop. 3.3]{kechris1995} we know that $\Theta=\mathbb{R}\times\overline{\P}$ is Polish. \cref{asp:bdd,asp:lscD} are satisfied by assumption, and that $L(X_1^n,\,\theta)$ are $\sigma(X_1^n)\times\mathcal{T}$ measurable is immediate, hence \cref{asp:loss} holds. Since $\mathcal{Q}=\PTheta$, we trivially include Dirac measures and \cref{asp:finite} holds.

We find it more intuitive to work with a log likelihood ration in the remainder rather than the negative log likelihood. This is equivalent through addition of a constant term depending on $n$ without affecting any minimisers. 
Let
$$
\begin{aligned}
L(X_1^n, (\theta,g)) :=-\frac{1}{n}\log\frac{p(X_1^n|\theta,g)}{p_0(X_1^n)}\rightarrow &\mathbb{E}_{P_0}\left[-\log\frac{p(X|\theta,g)}{p_0(X)}\right] \\&=D_{KL}(P_0:\delta_\theta\star G)=:L(\theta,g).
\end{aligned}
$$
As $p_0(X_1^n)\in(0,\infty)$ almost surely and $-n^{-1}\log p(X_1^n|\theta,g)\in [0,\infty]$, the losses are lower bounded (a.s.). These converge pointwise $P_0$--a.s. by the Asymptotic Equipartition Property as $P_0$ is stationary and ergodic under the usual shift \cite{algoet1988}. In particular, as the model is well specified, we may pick $\theta=\theta_0$ and $g\equiv h_0$ to achieve $L(X_1^n, (\theta_0,h_0))=0$ for each $n$, and hence $$\limsup_n\inf_{(\theta,g)\in\Theta}L(X_1^n,(\theta,g))\le 0\le \inf_{(\theta,g)\in\Theta}L(\theta,g).$$ We therefore satisfy \cref{asp:L-ae}.
Since we consider symmetric densities $g$, $h_0$ is continuous, and the space of bounded continuous functions is dense in $L^1(m)$, we may restrict ourselves to continuous densities. It follows that the convolution $p(X_i|\theta,g)=\delta_\theta\star g(X_i)$ is continuous and coercivity is immediate as $p(X_1^n|\theta,g)$ vanishes at the boundary. \cref{asp:lscL} is thereby satisfied.

It remains to be shown that $\forall A\subset\Theta$ for which $L(A)>L(\Theta)$ it holds that $\liminf L_n(A)>L(\Theta)$. As the model is well specified $L(\Theta)=0$ which holds iff $P_0=P_\theta$ in the sense of measures. Note that the set $N_\eta^c$ is such a set as $h_0(\theta)$ has a strict global maximum at 0. For the conclusion it is sufficient to verify the condition for all $\eta>0$. Note $p(X_1^n|\theta,g)=\prod_{i=1}^ng(X_i-\theta)$. Now since the pairs $(\theta,g)$ are symmetric densities centred at $\theta$, and restricting $|\theta|\ge \eta$ it is immediate that the condition holds, as eventually $P_0$--a.s. for all $n$ sufficiently large $p_0(X_1^n)>\prod_ig(X_i-\theta)$ for a.e. $\theta\in N_\eta$.
\end{proof}

\subsection{The Choice of Divergence}\label{sec:divs}
Many choices of divergence satisfying \cref{asp:bdd,asp:lscD} exist.
The prime example of such divergences, that moreover satisfies \cref{asp:bdd} for unrestricted choices of $\mathcal{Q}$ and $\mathcal{G}$, is the Total Variation Distance defined as:
\begin{equation}
D_{TV}(Q:\Pi):=\sup_{A\in\mathcal{T}}|QA-\Pi A|\le 1.
\end{equation}
Since $Q$ and $\Pi$ are probability measures, which are by definition non--negative, the upper bound is achieved for instance in the case when $Q$ and $\Pi$ are mutually singular measures.

As an immediate consequence, the Hellinger Distance is also bounded above. By a well known result, see for instance \cite[Lemma B.5]{pinski2015},
\begin{equation}
D_H^2(Q:\Pi)=\frac12\int_\Theta\left(\sqrt{\frac{dQ}{d\mu}(\theta)}-\sqrt{\dfrac{d\Pi}{d\mu}(\theta)}\right)^2\mu(d\theta)\le D_{TV}(Q:\Pi).
\end{equation}

Furthermore, such a divergence can be constructed as part of the class of $f$--divergences \cites{ali1966, amari2016}. We follow \cite{amari2016} and define these as
\begin{equation*}
 D_f(Q:\Pi)=\mathbb{E}_{Q}\left[f\left(\frac{d\Pi}{dQ}\right)\right]
\end{equation*}
for some convex function $f\in C^1(\mathbb{R}_{\ge0},\mathbb{R})$, such that $f(1)=0$.
This is a vast class of divergences and contains for instance the KL and the $\chi^2$ divergences.

Making use of a result in \cite{amari2016}, we know that for $f$--divergences, there is an upper bound whenever the following limit exists and is finite:
\begin{equation}
    0\le D_f(Q:\Pi)\le\lim_{u\rightarrow0^+}\left\{f(u)+ uf\left(\frac1u\right)\right\}
\end{equation}
Note that this is independent of any choice of probability measures $Q$ and $\Pi$, and solely reliant on the function $f$.

For instance, we can pick the Le Cam $f$--divergence, which uses 
$$
f(u)=\frac{(u-1)^2}{u+1}
$$
for which this limit exists and, through basic analysis, is 2.

Furthermore, \cref{thm:n} for instance allows for the following divergence:
\begin{equation*}
	D_{TV}^{(M\sqrt{n})}(Q:\Pi):=n^{1/2}M||Q-\Pi||_{TV},
\end{equation*}
where $M>0$ is some sufficiently large constant, preventing over--concentration around the data in cases where the prior is not severely misspecified while maintaining the results on rates of convergence. We may hence achieve an interplay between Bayesian and frequentist inference.
\section{Proofs}\label{sec:proofs}
\subsection{Existence and Uniqueness of a Generalised Variational Inference Posterior}\label{sec:exists}

\begin{lemma}\label{lem:tightnessJ}
	Make \cref{asp:loss,asp:finite,asp:lscL}. Suppose $\mathcal{Q}$ is closed with respect to weak$^\star$ convergence.
	Pick any $C\in\mathbb{R}$, then the subspace $\mathcal{A}:=\{\nu\in\mathcal{Q}:J_{L_n}(\nu)\le C\}$ is tight. 
\end{lemma}

\begin{proof}
Fix $\varepsilon>0$, and we show that $\exists K\subset\Theta$ compact, depending only on $\varepsilon$ and the loss, such that $\forall\nu\in\mathcal{A}$ it holds $\nu K>1-\varepsilon$.

We may decompose the functional, for a compact set $K\subset\Theta$ to be specified later, into
$$
J_{L_n}(\nu)=\int_KL_n\,d\nu+\int_{K^c}L_n\,d\nu\ge \nu K\inf_{\theta\in K}L_n(\theta)+\nu K^c\inf_{\theta\in\Theta\backslash K}L_n(\theta)\,.
$$
As the loss is lower bounded and lower semi--continuous it achieves a minimum on any compact set, so $\inf_{\theta\in K}L_n(\theta)=c\in\mathbb{R}$ is achieved for some $\theta$ and exists, and hence by the lower bound, there exists some global minimum $c^\star:=\inf_{\theta\in\Theta} L_n(\theta)$. Now, take $c^\star\wedge0:=\min\{c^\star,0\}$, so that
$$
\nu K\inf_{\theta\in K}L_n(\theta)+\nu K^c\inf_{\theta\in\Theta\backslash K}L_n(\theta)\ge (c^\star\wedge0)+\nu K^c\inf_{\theta\in\Theta\backslash K}L_n(\theta)\,.
$$
Suppose the contrapositive that $\mathcal{A}$ is not tight, then we show that the hypothesis $J_{L_n}(\nu)\le C$ also fails. In particular, this implies $\nu K^c\ge\varepsilon$, so 
$$
J_{L_n}(\nu)\ge (c^\star\wedge0)+\varepsilon\inf_{\theta\in\Theta\backslash K}L_n(\theta)\,.
$$
And as the set $[c^\star\wedge0, \frac{C-(c^\star\wedge0)}{\varepsilon}]\subset\mathbb{R}$ is compact, by the coerciveness of the loss, there exists $K'\subset\Theta$ compact such that $L_n:\Theta\backslash K'\rightarrow\mathbb{R}\backslash[c^\star\wedge0, \frac{C-(c^\star\wedge0)}{\varepsilon}]$, in words there exists some compact set $K'$ such that the loss defined on its complement maps $\theta$ exclusively outside of $[c^\star\wedge0, \frac{C-(c^\star\wedge0)}{\varepsilon}]$. Note that $K'$ does not depend on $\nu$.
We may now let $K=K'$ to see that if $\mathcal{A}$ is not tight then $\exists\nu\in\mathcal{A}$ for which $J_{L_n}(\nu)>C$.
As $\varepsilon$ was arbitrary, the claim follows.
\end{proof}

\begin{proof}[Proof of \cref{thm:existenceGVI}]
As the loss is lower bounded, we may let this constant be $c:=\inf_{\theta}L(X_1^n,\,\theta)$. Suppose without loss of generality that this constant is negative. Then we claim that $\forall C>c$ the sublevel sets of the GVI functional $T_n(\cdot)$ are weak$^\star$ compact. As the loss and the divergence are lower semi--continuous, it follows that $\nu\mapsto J_{L_n}(\nu)$ is lower semi--continuous and so $\nu\mapsto T_n(\nu):= J_{L_n}(\nu)+D(\nu:\Pi)$ is lower semi--continuous \cites{ambrosio2008,dembo2010}. Hence, by Prokhorov's theorem we need to show that sublevel sets are tight \cite{billingsley1999}. This is immediate if $\Theta$ is compact, so we now show the result for $\Theta$ that are not.

Define $\mathcal{A}:=\{\nu\in\mathcal{Q}:T_n(\nu)\le C\}$, then we aim to show that $\forall\varepsilon>0$ $\exists K\in\mathcal{T}$ compact, depending only on $\varepsilon$, such that $\forall\nu\in\mathcal{A}$ $\nu K>1-\varepsilon$. 

By \cref{lem:tightnessJ} we have that the sublevel sets of the functional $J_{L_n}(\cdot)$ are tight which is sufficient to show that $\mathcal{A}$ is tight. 

As the divergence is non--negative, and there is some element $\hat{\nu}\in\mathcal{Q}$ for which $T_n(\hat{\nu})<\infty$ so $D(\hat{\nu}:\Pi)<\infty$, and for sufficiently large $C$, $\mathcal{A}$ is non--empty. Hence $\forall\nu\in\mathcal{A}$ we have that $C\ge T_n(\nu)\ge J_{L_n}(\nu)$. By \cref{lem:tightnessJ} the set ${\mathcal{A}}':=\{\nu\in\mathcal{Q}:J_{L_n}(\nu)\le C\}$ is tight, and by the previous $\mathcal{A}\subset\mathcal{A}'$ so $\mathcal{A}$ is tight. By Prokhorov's theorem, and the lower semi--continuity of the functional $T_n(\cdot)$, $\mathcal{A}$ is compact with respect to the weak$^\star$ topology.

Having shown that the sublevel sets of the GVI functional are weak$^\star$ compact, and as the functional is lower semi--continuous, we may now by the definition of the infimum, construct a sequence $(\ell_k)_{k\in\mathbb{N}}\subset\mathbb{R}$ converging downwards to $\ell^\star:=\inf_{\nu\in\mathcal{Q}}T_n(\nu)$ defined via the corresponding elements $(\nu_k)_k\subset\mathcal{A}$, i.e. $\ell_k:=T_n(\nu_k)$. Since both the divergence and the loss are lower bounded, so is the GVI functional, so $\ell^\star>-\infty$. By the claim we find that $(\nu_k)_k$ is a tight sequence of probability measures, so after passing to a subsequence (without relabelling) we must have that $\nu_k\overset{w^\star}{\rightarrow}\nu_\star$ for some $\nu_\star\in\mathcal{Q}$, and for which it holds that 
$$
\liminf_kT_n(\nu_k)\ge T_n(\nu_\star)$$
by the lower semi--continuity. As $T_n(\nu_k)\rightarrow\ell^\star$, $\liminf_kT_n(\nu_k)=\ell^\star$, and as $T_n(\nu_\star)\ge \ell^\star$, we conclude that $T_n(\nu_\star)=\ell^\star$ which is a global minimum of the GVI functional, so such a minimiser $\nu_\star\in\mathcal{Q}$ of \cref{eqn:gvi_min} exists.

For the uniqueness part, suppose that there exist two minimisers, $\nu_1$ and $\nu_2$. Then, as $\mathcal{Q}$ is convex we may pick $a\in(0,1)$ and let $\nu:=a\nu_1+(1-a)\nu_2$ such that
\begin{align*}
T_n(\nu)&=J_{L_n}\bigl(a\nu_1+(1-a)\nu_2\bigr)+D\bigl(a\nu_1+(1-a)\nu_2:\Pi\bigr)\\
&< aJ_{L_n}\bigl(\nu_1\bigr) + (1-a)J_{L_n}\bigl(\nu_2\bigr) + aD(\nu_1:\Pi)+(1-a)D(\nu_2:\Pi)\\
&= a T_n(\nu_1)+ (1-a)T_n(\nu_2)
\end{align*}
where the inequality follows by the convexity of the divergence.
But as $\nu_1$ and $\nu_2$ are both minimisers $T_n(\nu_1)=T_n(\nu_2)$ so $\nu$ would further lower the GVI objective, so these cannot be minimisers. Hence, we conclude that the only way $\nu=a\nu_1+(1-a)\nu_2$ is if $\nu_1=\nu_2$ (in the sense of measures), and the minimiser may not be written as a convex combination of elements of $\mathcal{Q}$ so this minimiser is an extreme point and unique.
\end{proof}

\subsection{Characterising the GVI Posterior through the Empirical Loss Minimiser}\label{sec:proofset}

\begin{lemma}\label{lem:pref}
Let \cref{asp:loss,asp:finite,asp:bdd} hold. Suppose $\exists C\in\mathbb{R}$ and that $\exists P,Q\in\mathcal{Q}$ such that 
\begin{align*}
	\mathbb{E}_{P}[L(X_1^n,\,\theta)]\le C,\; \mathrm{and}\; \mathbb{E}_{Q}[L(X_1^n,\,\theta)]> C+\frac{M}{n\beta},
\end{align*}%
then $\forall\Pi\in\mathcal{G}$ it holds that $T_n(P)<T_n(Q)$.
\end{lemma}
\begin{proof}%
For any $P,Q\in\mathcal{Q}$ satisfying the assumption above, we have that:
\begin{gather*}
T_n(P)=n\mathbb{E}_{P}[L(X_1^n,\theta)]+\frac1\beta D(P:\cdot)\le nC+ \sup_{\Pi\in\mathcal{G}}\frac1\beta D(P:\Pi)\le nC + \frac M{\beta} \\
T_n(Q)=n\mathbb{E}_{Q}[L(X_1^n,\theta)]+\frac1\beta D(Q:\cdot)>n(C+\frac M{n\beta})+ \inf_{\Pi\in\mathcal{G}}\frac1\beta D(Q:\Pi)\ge nC+\frac M{\beta}
\end{gather*}
hence 
$T_n(P) < T_n(Q)$ for all priors $\Pi\in\mathcal{G}$.
\end{proof}

\begin{proof}[Proof of \cref{thm:set}]
By \cref{asp:mle}, $\exists P_n^\star\in\mathcal{Q}$ such that
$$
P_n^\star\in\argmin_{Q\in\mathcal{Q}}\mathbb{E}_Q[L(X_1^n,\,\theta)].
$$
Now let $C= \mathbb{E}_{P_n^\star}[L(X_1^n,\,\theta)]$. 
Then for all $Q\in\mathcal{Q}$ that satisfy 
$\mathbb{E}_{Q}[L(X_1^n,\,\theta)]>C+ M(n\beta)^{-1},$ 
by \cref{lem:pref} it must hold that $T_n(P_n^\star)<T_n(Q)$, and hence $Q(d\theta)$ is not a minimiser of the GVI objective. 
Therefore, the only measures $Q'\in\mathcal{Q}$ that are potential minimisers of the GVI objective for any prior $\Pi\in\mathcal{G}$ are those that satisfy: $T_n(Q')\le T_n(P_n^\star)$.

And since $T_n(P_n^\star)\le nC+\beta^{-1}M$, we have
\begin{gather*}
    T_n(Q')=n\mathbb{E}_{Q'}[L(X_1^n,\,\theta)]+\beta^{-1}D(Q':\Pi)\le nC+\beta^{-1}M\\
    \implies n(\mathbb{E}_{Q'}[L(X_1^n,\,\theta)]- C)\le \beta^{-1}(M - D(Q':\Pi))\\
    \implies \mathbb{E}_{Q'}[L(X_1^n,\,\theta)]- \mathbb{E}_{P_n^\star}[L(X_1^n,\,\theta)]\le M(n\beta)^{-1}
\end{gather*}
Therefore, all probability measures $Q'\in\mathcal{Q}$ that satisfy the above bound are potential minimisers of $T_n(\cdot)$. This is precisely the set $\mathcal{R}^\star_n$. Furthermore, by \cref{asp:mle} we know that $\forall Q'\in\mathcal{Q}$ we have $\mathbb{E}_{Q'}[L(X_1^n,\,\theta)]\ge C$, so the LHS is non--negative, and since $P_n^\star$ is in this set, it is non--empty.

For the moreover part, we additionally suppose that $\mathcal{Q}$ is weak$^\star$ closed and that the loss is lower semi--continuous and coercive. It immediately follows that the functional $\nu\mapsto J_{L_n}(\nu)$
is lower semi--continuous, note that this follows by the functional being over probability measures and the loss being lower bounded and lower semi--continuous \cite{ambrosio2008}. 
We then know  by \cref{lem:tightnessJ} that the sublevel sets of the functional $J_{L_n}(\cdot)$ are tight. Hence, by the lower semi--continuity of the functional, $\mathcal{Q}$ being weak$^\star$ closed, and Prokhorov's theorem \cite{billingsley1999}, the set $\mathcal{A}:=\{\nu\in\mathcal{Q}:J_{L_n}(\nu)\le k\}$ is weak$^\star$ compact.

Therefore, we may apply this to $\mathcal{R}_n^\star$ by picking $k=M(n\beta)^{-1}+\inf_{\nu\in\mathcal{Q}}J_{L_n}(\nu)$, which exists and is finite by assumption, establishing that $\mathcal{R}_n^\star$ is weak$^\star$ compact as desired.
\end{proof}

\subsection{Asymptotics and Frequentist Consistency}\label{sec:asymptotics}

\begin{lemma}\label{lem:minF}
Suppose that $f:\Theta\rightarrow\mathbb{R}$ is a $\mathcal{T}$--measurable, lower--bounded, lower semi--continuous, and coercive function.
Any minimiser of the functional $J_f(\cdot)$ within $\PTheta$ must put full measure on any measurable set containing the minimisers of the function $f(\theta)$. 
\end{lemma}
\begin{proof}
Let $\varepsilon>0$ and denote $N_\varepsilon:=\{\theta\in\Theta:f(\theta)\le\inf_{\vartheta\in\Theta}f(\vartheta)+\varepsilon\}$. For ease of notation, we denote $\inf_{\theta\in A} f(\theta)\equiv f(A)$.
We then note that for $\nu\in\PTheta$
$$
J_f(\nu):=\int_\Theta fd\nu=\int_{N_\varepsilon}fd\nu+\int_{N_\varepsilon^c}fd\nu\ge f(\Theta)\,.
$$
As the function is lower bounded, lower semi--continuous and coercive, there exists some minimiser away from the boundary of $\Theta$, specifically it achieves its minimum on some compact subset of $\Theta$ away from the boundary, so $N_0:=\{\theta\in\Theta:f(\theta)=f(\Theta)\}$ is well defined and non--empty. Note that this is measurable as we may define it as a countable intersection of measurable sets.

Then for any measure $\nu\in\PTheta$ for which $\nu N_0=1$ the lower bound above is achieved. Hence, $\forall\varepsilon>0$ and $\forall\mu\in\PTheta$ for which $\mu N_\varepsilon<1$, it follows that $J_f(\nu)<J_f(\mu)$. We may see this as 
\begin{align*}
J_f(\mu)=\int_{N_\varepsilon}fd\mu+\int_{N_\varepsilon^c}fd\mu    
&\ge\mu N_\varepsilon\cdot f(\Theta) + \mu N_\varepsilon^c \cdot (f(\Theta)+\varepsilon)
\\&
=f(\Theta)+\varepsilon(1-\mu N_\varepsilon)>f(\Theta)
\end{align*}
establishing the claim.
\end{proof}

\begin{lemma}\label{lem:LnTheta}
Make \cref{asp:loss,asp:finite,asp:lscL}, and let $\mathcal{Q}=\PTheta$. Then $(P^\star_n)_n$ will be a sequence of Dirac measures placing full measure on the set of points minimising the corresponding loss. As a consequence we have that  $J_{L_n}(P_n^\star)=L_n(\Theta)$.
\end{lemma}

\begin{proof}
By \cref{lem:minF} we have that the measure minimising $J_{L_n}(\cdot)$ will have full measure on the set $N^n_0:=\{\theta\in\Theta:L(X_1^n,\,\theta)=L_n(\Theta)\}$. Since any such measure achieves the lower bound on the functional $J_{L_n}(\cdot)$ we may pick any $\theta_n^\star\in N^n_0$ and let $P^\star_n:=\delta_{\theta^\star_n}$. Then $\delta_{\theta^\star_n}N_0^n=1$ as required.
\end{proof}

\begin{lemma}\label{lem:concentrationMLE}
Make \cref{asp:loss,asp:finite,asp:lscL,asp:L-ae} and suppose $\mathcal{Q}=\PTheta$. Let $A\in\mathcal{T}$ be any set for which $L(A)>L(\Theta)$. Suppose further that for any such $A$ it almost surely holds that
$$
\liminf_{n}\;L_n(A)>L(\Theta)\,,
$$
then $P_0$ almost surely $P_n^\star A\rightarrow 0$ as $n\rightarrow\infty$.
\end{lemma}
\begin{proof}
For any such set $A$, there exists an $\varepsilon>0$ such that the sets $N_\varepsilon$ and $A$ are disjoint.
In particular, by \cref{lem:minF} $\forall n\in\mathbb{N}$ $P_n^\star N_\varepsilon^n=1$ $\forall\varepsilon>0$. Therefore, it remains to be shown that eventually almost surely $\exists\varepsilon>0$ such that $A\cap N_\varepsilon^n=\emptyset$ for all $n$ sufficiently large which will imply the claim. 
To see this, by the assumption it is almost sure that $
\liminf_n L_n(A)>L(\Theta)$
which implies we may pick $\varepsilon=1/2(\liminf_nL_n(A)-L(\Theta))>0$. Further, $L_n\rightarrow L$ as $n\rightarrow\infty$ pointwise, and %
$\limsup_nL_n(\Theta)\le L(\Theta)$ ($P_0$--a.s.) by \cref{asp:L-ae}. Hence for all $n$ sufficiently large it holds almost surely that $L_n(A)-L_n(\Theta)>0$. And hence, eventually $P_0$--a.s. $A$ and $N_\varepsilon^n$, with $\varepsilon$ as above, will be disjoint.
\end{proof}

\begin{lemma}\label{lem:concentration}
Make \cref{asp:loss,asp:lscL,asp:bdd,asp:finite,asp:lscD,asp:L-ae}, and suppose that $\mathcal{Q}=\PTheta$. Recall the definition of the set $N_\varepsilon:=\{\theta\in\Theta:L(\theta)\le L(\Theta)+\varepsilon\}$. Suppose that $\forall\varepsilon>0$ the losses $P_0$--a.s. satisfy
$$
\liminf_{n}\,L_n(N_\varepsilon^c)>L(\Theta)\,.
$$
Then $Q_nN_\varepsilon\rightarrow1$ as $n\rightarrow\infty$ $P_0$--a.s. for any $\varepsilon>0$.
\end{lemma} 
\begin{proof}
We know that $P_n^\star N_\varepsilon\rightarrow1$ as $n\rightarrow\infty$ $P_0$--a.s. $\forall\varepsilon>0$ by \cref{lem:concentrationMLE}. By \cref{thm:set} we have that $\forall n$ the GVI posteriors $Q_n\in\mathcal{R}{_n^\star}:=\{\nu\in\mathcal{Q}:J_{L_n}(\nu)-J_{L_n}(P^\star_n)\le\frac M{\beta n}\}$.

Fix $\varepsilon>0$. Recall that by \cref{lem:LnTheta} we have $J_{L_n}(P_n^\star)=L_n(\Theta)$, and by \cref{asp:L-ae} $\limsup_nL_n(\Theta)\le L(\Theta)$ (a.s.). By assumption we have that eventually $L_n(N_\varepsilon^c)>L(\Theta)$ a.s. for all $n$ sufficiently large and therefore $ \liminf_nL_n(N_\varepsilon)=\liminf_n L_n(\Theta)$ $P_0$--a.s. 

By \cref{thm:set} we may write for all $n\in\mathbb{N}$
\begin{align}
    \frac{M}{\beta n}&\ge J_{L_n}(Q_n)-J_{L_n}(P^\star_n)\nonumber\\
    &=\int_{N_\varepsilon}L_ndQ_n+\int_{N_\varepsilon^c}L_ndQ_n-L_n(\Theta)\nonumber\\
    &\ge Q_nN_\varepsilon\cdot L_n(N_\varepsilon)+Q_nN_\varepsilon^c\cdot L_n(N_\varepsilon^c)-L_n(\Theta)\nonumber\\
    &=Q_nN_\varepsilon\left( L_n(N_\varepsilon)-L_n(\Theta)\right)+Q_nN_\varepsilon^c\left(L_n(N_\varepsilon^c)-L_n(\Theta)\right)\ge0\,.\label{eqn:concentration}
\end{align}
This is strictly non--negative $\forall n$ as $L_n(A)\ge L_n(\Theta)$ for all sets $A\in\mathcal{T}$. Then eventually $P_0$--a.s., $L_n(N_\varepsilon^c)>L(\Theta)$, and it suffices to show the behaviour of ${Q_n}$ on $N_\varepsilon^c$. 
Therefore, $P_0$ almost surely
\begin{align*}
    L_n(N_\varepsilon)-L_n(\Theta)\ge L_n(\Theta)-L_n(\Theta)= 0\,, \qquad \liminf_nL_n(N_\varepsilon^c)-L_n(\Theta)>0\,.
\end{align*}

As $M(\beta n)^{-1}$ converges to 0 uniformly, and we have the non--negativity of \cref{eqn:concentration}, we conclude that $Q_nN_\varepsilon^c\rightarrow0$ as $n\rightarrow\infty$ (a.s.), establishing the claim.
\end{proof}

\begin{proof}[Proof of \cref{thm:asymptotics}]
Fix $A\in\mathcal{T}$ for which $L(A)>L(\Theta)$, and pick $\varepsilon=1/2(L(A)-L(\Theta))>0$ for which it holds that $A\subset N_\varepsilon^c$. 

Since $\mathcal{Q}\subset\PTheta$ allows for Dirac measures we have that for all $n$ sufficiently large $\inf_{\nu\in\mathcal{Q}}J_{L_n}(\nu)=L_n(\Theta)$. By \cref{lem:tightnessJ} and reasoning similar to the proof of \cref{thm:existenceGVI} for each $n$ $\exists(\nu^n_k)_k\subset\mathcal{Q}$ s.t. after passing to a subsequence in $k$ (without relabelling), %
$\nu_k^n\overset{w^\star}{\rightarrow}\nu_\star^n\in\mathcal{Q}$ as $k\rightarrow\infty$ for which $J_{L_n}(\nu_\star^n)=L_n(\Theta)$ $P_0$--a.s. 
It follows \cref{lem:concentration} applies and hence $Q_nN_\varepsilon\rightarrow1$ (a.s.) so $Q_nA\rightarrow0$ (a.s.).

For the moreover part, we assume that $\exists \theta^\star\in\Theta$ uniquely minimising $L:\Theta\rightarrow\mathbb{R}$. Therefore, the set $N_0:=\{\theta\in\Theta:L(\theta)=L(\Theta)\}$ is a singleton namely $N_0=\{\theta^\star\}$ and hence $\delta_{\theta^\star}$ minimises the functional $J_{L}(\cdot)$. As per the previous $Q_nN_\varepsilon\rightarrow1$ $\forall\varepsilon>0$ $P_0$--a.s. we may establish the following for sets $A\in\mathcal{T}$:
\begin{enumerate}
    \item If $\theta^\star\in A$ and $L(A)=L(\Theta)$, then $\delta_{\theta^\star}A=1$ so $Q_nA\rightarrow1$, and
    \item If $\theta^\star\notin A$ and $L(A)>L(\Theta)$, then $\delta_{\theta^\star}A=0$ so $Q_nA\rightarrow0$
\end{enumerate}
which hold by \cref{lem:concentration}. However, for sets like $A=\Theta\backslash\{\theta^\star\}\in\mathcal{T}$ we have $L(A)=L(\Theta)$. The weak$^\star$ convergence now follows by a typical $3\varepsilon$ proof. Recall that simple functions are dense in $C_b(\Theta)$ so for any $f\in C_b(\Theta)$ and any $\varepsilon>0$ there exists a simple function $g:\Theta\rightarrow\mathbb{R}$, where $A_j$ is a pairwise disjoint partition of $\Theta$ such that $g:=\sum_j\alpha_j\mathds{1}_{A_j}$ for some constants $\alpha_j\in\mathbb{R}$, such that we may approximate $f$ uniformly within $\varepsilon$ through $g$. It follows
$$
\left|\int_\Theta f \,dQ_n-\int_\Theta f\,d\delta_{\theta^\star}\right|\le 2\varepsilon + \sum_j|\alpha_j|\,|Q_n(A_j)-\delta_{\theta^\star}(A_j)|
$$
We may pick $(A_j)$ such that $\exists k$ for which $\delta_{\theta^\star}\in A_k$, and $\forall j\ne k$ $L(A_j)>L(\Theta)$. Hence, the above applies and for all $n$ sufficiently large we have that the entire expression is eventually $P_0$--a.s. less than $3\varepsilon$. As $\varepsilon$ was arbitrary the result follows. 
\end{proof}

\subsection{Consistency under Arbitrary Divergences}
\begin{proof}[Proof of \cref{thm:unboundedConcentration}]
By \cref{thm:existenceGVI} $Q_n$ exists $\forall n$. Then, for any such sequence $(\nu_n)\subset\mathcal{Q}$ as specified in the assumption it follows that $P_0$ almost surely
$$n^{-1}T_n(\nu_n)-L_n(\Theta)=J_{L_n}(\nu_n)-L_n(\Theta)+(n\beta)^{-1}D(\nu_n:\Pi)\rightarrow 0\,.$$
Pick any such set $A\in\mathcal{T}$ for which $L(A)>L(\Theta)$. It follows that $P_0$--a.s.
\begin{align*}
\frac1nT_n(\nu_n)\ge\frac1nT_n(Q_n)\ge J_{L_n}(Q_n) &=\int_AL_n(\theta)Q_n(d\theta)+\int_{A^c}L_n(\theta)Q_n(d\theta)\\
&\ge Q_nA\:L_n(A)+Q_nA^c\,L_n(A^c)\ge L_n(\Theta)\,.
\end{align*}
It follows that $n^{-1}T_n(Q_n)-L_n(\Theta)\rightarrow 0$. This implies
\begin{equation*}
    J_{L_n}(Q_n)-L_n(\Theta)\rightarrow0\,,\quad\mathrm{and}\quad (n\beta )^{-1}D(Q_n:\Pi)\rightarrow0
\end{equation*}
as both quantities are non--negative. Therefore, $J_{L_n}(Q_n)-L_n(\Theta)\ge Q_nA(L_n(A)-L_n(\Theta))\ge0$. But as 
\begin{gather*}
\liminf_{n\rightarrow\infty}L_n(A)>L(\Theta)\,,\quad \mathrm{and}\quad \limsup_{n\rightarrow\infty}L_n(\Theta)<L(\Theta)\\
\implies\liminf_{n\rightarrow\infty}(L_n(A)-L_n(\Theta))\ge \liminf_{n\rightarrow\infty}L_n(A)-\limsup_{n\rightarrow\infty}L_n(\Theta) >0\,,
\end{gather*}
and the result follows.
\end{proof}

\subsection{Generalisation Performance}
\begin{proof}[Proof of \cref{thm:gen}]
The proof follows exactly the same steps as Theorem 5 in \cite{shalizi2009} so we do not provide it here.
\end{proof}
\subsection{Rate of Convergence}\label{sec:rates}
\begin{proof}[Proof of \cref{thm:rates}] We may show this by demonstrating that $Q_nN_{\varepsilon_n}^c\rightarrow0$.
By the proof of \cref{thm:asymptotics} we may apply \cref{eqn:concentration} of \cref{lem:concentration}, such that $\forall n\in\mathbb{N}$ sufficiently large
\begin{align*}
    \frac{M}{\beta}n^{-1} &\ge J_{L_n}(Q_n)-J_{L_n}(P_n^\star)\\
    &\ge Q_n N_{\varepsilon_n}(L_n(N_{\varepsilon_n})-L_n(\Theta))+ Q_n N_{\varepsilon_n}^c (L_n(N_{\varepsilon_n}^c)-L_n(\Theta))\ge0\,.
\end{align*}
Further we have that
$$
N_{\varepsilon_n}:=\{\theta\in\Theta:L(\theta)\le L(\Theta)+\varepsilon_n\}\rightarrow\{\theta\in\Theta:L(\theta)=L(\Theta)\}=:N_0
$$
which is well defined as the loss is lower semi--continuous.
Then, by assumption on the loss and as it suffices to show the behaviour of $Q_n$ on $N_{\varepsilon_n}^c$, it follows that $L_n(N_{\varepsilon_n})\ge L_n(\Theta)$ and 
$$L_n(N_{\varepsilon_n})-L_n(\Theta)\ge L_n(\Theta) -L_n(\Theta)=0\,.$$
So we may discard the term $Q_n N_{\varepsilon_n}(L_n(N_{\varepsilon_n})-L_n(\Theta))$ to get
$${M}{\beta}^{-1}n^{-1}  \ge Q_n N_{\varepsilon_n}^c (L_n(N_{\varepsilon_n}^c)-L_n(\Theta))\ge 0\,.
$$
We now need to show that $Q_nN_{\varepsilon_n}^c\rightarrow0$. 
Taking these together, eventually $P_0$--a.s. $n$ is large enough for $\varepsilon_n$ to be sufficiently small such that we have $L_n(N_{\varepsilon_n}^c)-L_n(\Theta)\gtrsim \varepsilon_n$. Therefore it follows that 
\begin{align*}
    \frac M\beta n^{-1} &\gtrsim \varepsilon_nQ_nN_{\varepsilon_n}^c\\
    \frac M\beta&\gtrsim n\varepsilon_n\cdot Q_nN_{\varepsilon_n}^c\,.
\end{align*}
Then, as $n\varepsilon_n\rightarrow\infty$, it must follow that $P_0$--a.s. $Q_nN_{\varepsilon_n}^c\rightarrow0$ as otherwise, eventually 
$$
n\varepsilon_n\gtrsim\frac M{\beta \cdot Q_nN_{\varepsilon_n}^c}
$$
$P_0$ almost surely, which would contradict $Q_n\in\mathcal{R}_n^\star$. 
\end{proof}

\subsection{Dependence on $n$}
\begin{proof}[Proof of \cref{thm:n}]
    This is immediate from the preceding proofs of \cref{thm:asymptotics,thm:rates}.
\end{proof}
\section{Concluding Remarks}\label{sec:discussion}
In this paper we have examined the asymptotic behaviour of Generalised Variational Inference measures on infinite dimensional hypothesis spaces with particular focus on prior misspecification. We have established frequentist consistency and rates of convergence to minimisers of $L(\theta)$ when using bounded divergences allowing for consistency regardless of the prior measure $\Pi\in\mathcal{G}$; see \cref{fig:misspec}. In fact, supposing that the divergence $D$ is everywhere bounded enables the choice $\mathcal{G}=\PTheta$ hence including priors that are severely misspecified and result in inconsistent Bayesian posteriors. 

It is worth noting that these results have further impact to Federated GVI \cite{mildner2025}, the study of GVI posteriors over a network of clients with private, local data sets. In particular such \FedGVI posteriors will exist under the assumptions presented here through the different choices of loss functions and priors. Moreover, a global prior of some uniformed server, given a bounded divergence, may not skew the global posterior after incorporating local data towards unfavourable hypotheses that do not agree.

Finally, while we were able to derive existence and uniqueness, as well as a consistency result for GVI posteriors under unbounded divergences, rates of convergence of GVI posterior measures under unbounded divergences remains an open challenge.

\section*{Acknowledgement}
TM and TD were supported by a UKRI Turing AI Acceleration Fellowship [EP/V02678X/1] and a Turing Impact Award from The Alan Turing Institute. The first author would like to direct correspondence to \href{mailto:Terje.Mildner@kellogg.ox.ac.uk}{Terje.Mildner@kellogg.ox.ac.uk} or \href{mailto:Terje.Mildner@gmail.com}{Terje.Mildner@gmail.com}.

\appendix

\section{General}
As we expect this paper to be of interest also for researchers who are not as familiar with the intricacies of measure theory we recall some basic facts and definitions used throughout the paper without explicit statements.

We have used several modes of convergence of measures throughout the paper, we summarise these below.
\begin{definition}
	Let $(\mu_n)_n$ and $\mu$ be Borel probability measures on the measurable space $(\Theta,\mathcal{T})$. We have the following notions of convergence.
	\begin{enumerate}
		\item If $\int_\Theta fd\mu_n\rightarrow\int_\Theta fd\mu$ $\forall f\in C_b(\Theta)$, then we say $\mu_n\overset{w^\star}{\rightarrow}\mu$, i.e. in the weak$^\star$ topology.
		\item If $\mu_nA-\mu A\rightarrow0$ $\forall A\in\mathcal{T}$, then $\mu_n\overset{sw}{\rightarrow}\mu$, i.e. setwise.
		\item If $||\mu_n-\mu||_{TV}\rightarrow0$, then $\mu_n\overset{TV}{\rightarrow}\mu$, i.e. in total variation.
	\end{enumerate}
	Here $C_b(\Theta)$ denotes the space of bounded, continuous functions mapping $\Theta$ to $\mathbb{R}$.
	Note, convergence in total variation implies setwise convergence, and setwise convergence in turn implies weak$^\star$ convergence, however the converse do not hold in general.
\end{definition}
\begin{definition}
	A function $f:X\rightarrow Y$ between two topological spaces $(X,\mathcal{B})$ and $(Y,\mathcal{C})$ is called coercive if for every compact $K_Y\subset Y$ there exists a compact set $K_X\subset X$ such that the elements of $K_X^c$ map exclusively outside $K_Y$, that is 
    $$f(X\backslash K_X)\subseteq Y\backslash K_Y.$$
\end{definition}
\begin{definition}
	We say that $\mathcal{Q}$ `allows for Dirac measures' if (i) $\mathcal{Q}$ is closed with respect to weak$^\star$ convergence of probability measures, and (ii) if for all coercive, lower bounded and lower semi--continuous functions $f:\Theta\rightarrow\mathbb{R}$ it holds that
	$$\inf_{\nu\in\mathcal{Q}}\;J_f(\nu)= \inf_{\theta\in\Theta}\;f(\theta)\,.
	$$
\end{definition} 
Note that this is trivially achieved if all Dirac delta measures are included within $\mathcal{Q}$, but may equivalently be achieved for instance through the set of Gaussian measures by allowing for measures with arbitrary mean and zero variance \cite{pinski2015}. 
The most straightforward definition of such Gaussian measures on infinite dimensional spaces is the one on separable Hilbert spaces $(\mathcal{H}, (\cdot,\cdot),||\cdot||)$. Following \cite{pinski2015}, these measures are uniquely characterised by their mean $m\in\mathcal{H}$ and covariance operator $C:\mathcal{H}\rightarrow\mathcal{H}$. So we say a measure $\nu\in\mathcal{P}(\mathcal{H})$ is Gaussian if it satisfies
\begin{equation*}
	m:=\int_{\mathcal{H}}\theta \nu(d\theta)\,,\quad \mathrm{and}\quad(\theta_1,C\theta_2)=\int_{\mathcal{H}}(\theta,\theta_1)(\theta,\theta_2)\nu(d\theta)\quad\forall\theta_1,\theta_2\in\mathcal{H}
\end{equation*}
where the first is a Bochner integral. 
\section{Classical Robustness to Prior Misspecification}\label{apx:classicPriorRobust}
We briefly discuss the relation of our work to robustness to prior misspecification in the sense of Bayesian sensitivity analysis \cite{berger1985,gustafson1995}. 
This line of work has examined the stability of Bayesian posteriors under small perturbations of the prior with respect to some other measure, for instance through additive Huber contamination as mentioned in \cref{sec:prior_mis}. 
In particular, this has been used to select a robust prior for Bayes which leads to optimal decision making under potential perturbations of the prior measure. In this paper however, we aim to answer the reverse: under what conditions can GVI lead to optimal decision making under ill chosen priors.

Local sensitivity of posteriors generated with $\varepsilon$--contaminated priors, for instance by the additive Huber contamination model $\mathcal{G}=\{\Pi_\varepsilon:=(1-\varepsilon)\Pi_0+\varepsilon G: \varepsilon\in(0,1)\}$, is often studied \cite{gustafson1995} through the G\^{a}teux differential 
$$
s(\Pi_0, G;X_1^n)=\lim_{\varepsilon\downarrow0}\cfrac{d(Q^n_0:Q^n_\varepsilon)}{d(\Pi_0:\Pi_\varepsilon)}
$$
where $d(\cdot:\cdot)$ is some distance or divergence measure; Gustafson and Wasserman \cite{gustafson1995} for instance use the total variation distance.
Here $\Pi_0$ is the uncontaminated prior, $\Pi_\varepsilon$ the $\varepsilon$--contaminated prior in the direction of $G$, and $Q^n_0$ and $Q^n_\varepsilon$ the respective Bayes posteriors. 
The unrestricted choice of a divergence other than the Kullback--Leibler in GVI generally prohibits explicit expressions of $Q^n_0$ and $Q^n_\varepsilon$ that demonstrate the the exact dependence on the prior, which makes this approach impracticable for our purposes.

This presents an interesting theoretical question to the stability of Bayesian posteriors with respect to prior misspecification and suggests fascinating future directions for understanding uncertainty about the hypotheses in GVI given finite observations. 
The work in \cite{berger1985, gustafson1995} is concerned with small perturbations in the prior within a restricted class of measures $\mathcal{G}$. Often, restrictions are placed on this set in terms of how much the prior may vary; in \cite{dey1994} the priors considered are those close to the original prior $\Pi_0$ as $\varepsilon$ tends to zero. The present paper however is concerned with severe prior misspecification, even cases where we allow $\mathcal{G}=\PTheta$, as well as the asymptotic dynamics of GVI posterior measures under such misspecification.

\end{document}